\documentclass{amsart}

\usepackage{graphicx}

\newtheorem{theorem}{Theorem}
\newtheorem{lemma}[theorem]{Lemma}
\newtheorem{prop}[theorem]{Proposition}
\newtheorem{cor}[theorem]{Corollary}

\theoremstyle{definition}

\theoremstyle{remark}

\numberwithin{equation}{section}

\newcommand{\Gab}{G  -  a,b}
\newcommand{\maxd}[1]{\Delta(#1)}
\newcommand{\mind}[1]{\delta(#1)}
\newcommand{\CP}{\mathcal P}
\newcommand{\R}{\mathbb R}

\newcommand{\bibtitle}[1]{\emph{#1}}
\newcommand{\dfn}[1]{\textbf{#1}}
\newcommand{\ty}{\nabla\mathrm{Y}}
\newcommand{\yt}{\mathrm{Y}\nabla}

\begin{document}

\title{Order nine MMIK graphs}



\author{Thomas W.\ Mattman}
\address{Department of Mathematics and Statistics,
California State University, Chico,
Chico, CA 95929-0525}
\email{TMattman@CSUChico.edu}

\author{Chris Morris}
\address{
Computer Science Department, California State University, Chico,
Chico, CA 95929-0410}
\email{chris@chrismorris.net}
\author{Jody Ryker}
\address{Department of Mathematics,
University of California Santa Cruz, Santa Cruz, CA 95064}
\email{jryker@ucsc.edu}


\date{}

\begin{abstract}
We show that there are exactly eight MMIK (minor minimal intrinsically knotted) graphs 
of order nine.
\end{abstract}

\maketitle

\section*{Introduction}
A graph is \dfn{intrinsically knotted (IK)} if every tame embedding in $\R^3$ has a nontrivially
knotted cycle. Since the opposite property `not intrinsically knotted' is closed under 
taking minors~\cite{NT}, it follows from the Graph Minor Theorem of Roberston and Seymour~\cite{RS} that the set of IK graphs is characterized by a finite set of MMIK (minor minimal IK) graphs.
Recall that a \dfn{minor} of a graph $G$ is any graph obtained by contracting edges in a subgraph
of $G$. We say that graph $G$ is \dfn{minor minimal} 
with respect to graph property $\CP$ if $G$ has 
$\CP$, but no proper minor does.

In their seminal paper Conway and Gordon~\cite{CG} demonstrate 
that $K_7$ is IK and Kohara and Suzuki~\cite{KS} subsequently showed it
is MMIK.
This means $K_7$ is the unique MMIK graph on seven vertices
and any graph of lesser order is not IK. Two groups, working independently, classified the
MMIK graphs on eight vertices~\cite{BBFFHL, CMOPRW}; there are exactly two:
$K_{3,3,1,1}$ and the graph obtained by a single $\ty$ exchange on $K_7$.
We take the next step by proving the following.

\begin{theorem}
\label{thmain}%
There are exactly eight MMIK graphs of order nine.
\end{theorem}

The eight graphs have all been described elsewhere, including proofs that they 
are MMIK. In this paper we show there are no other examples. Two of the graphs
are in the $K_7$ family, first described by Kohara and Suzuki~\cite{KS}; they call 
those graphs $F_9$ and $H_9$ and proved that they are MMIK.

The remaining examples are described in~\cite{GMN}; in this paragraph we
summarize the relevant ideas from that paper. The \dfn{family} of a graph $G$ 
is the set of graphs that can be obtained from
$G$ by a sequence of $\ty$ and $\yt$ moves. The $K_{3,3,1,1}$ family
contains four MMIK graphs of order nine. Two of these come from $\ty$ moves
on $K_{3,3,1,1}$ and were known to be MMIK by combining work of Foisy~\cite{F}
with Kohara and Suzuki~\cite{KS}. Returning to~\cite{GMN}, the authors
show that the other two order nine graphs in that family are also MMIK. For this, they use
Miller and Naimi's~\cite{MN} program to argue that the two graphs are IK.
The final two MMIK graphs of order nine are called
$E_9+e$ and $G_{9,28}$ in \cite{GMN}. They have a traditional proof 
that $G_{9,28}$ is MMIK while again relying on the program of~\cite{MN}
to verify that $E_{9+e}$ is IK. In summary, Miller and Naimi's program is used to show that 
three of the order nine graphs are IK. Even in those cases, 
the authors provide a traditional proof that those graphs are minor minimal IK. 
See~\cite{GMN} for details.

Our proof that there are no other instances of order nine MMIK graphs also relies on 
computers. Let us first outline the parts of the argument
that are independent of the computer. In her thesis~\cite{R}, the third author proves
the following.
\begin{prop}
\label{prop28}%
The only order nine MMIK graph of size 28 or greater is $G_{9,28}$.
\end{prop}
The classification of MMIK graphs of size 21 or less is due, independently, 
to two groups and has the following corollary.
\begin{cor} \cite{BM,LKLO}
\label{cor21}%
There are exactly two order nine MMIK graphs of size at most 21, $F_9$ and 
$H_9$.
\end{cor}
In the current paper, we give a conventional proof for graphs on 22 edges.
\begin{prop}
\label{prop22}%
There are exactly five order nine MMIK graphs of size 22.
\end{prop}

We prove Propositions~\ref{prop28} and \ref{prop22} in Sections 2 and 3 respectively.
What remains is to prove that there are no order nine MMIK graphs with
between 23 and 27 edges. For this we rely on computers. 
As described in Section 4, we have two distinct
approaches implemented in three different computer languages.
We feel that this variety of techniques and languages all pointing to the same conclusion
amount to a rather robust proof of the remaining cases. Together with 
Corollary~\ref{cor21} and  Propositions~\ref{prop28} and \ref{prop22}, 
the argument of Section 4 completes our
proof of Theorem~\ref{thmain}. We begin by gathering preliminary definitions 
and lemmas in the next section.

\section{Definitions and Lemmas}
This section collects a number of definitions and lemmas used through the rest of the paper.
For a graph $G$, $|G|$ will be the order or number of vertices and $\|G\|$ the size or number
of edges, and we frequently use the pair
$(|G|, \|G\|)$ to describe a graph.
The maximum and minimum degree among the vertices of $G$ are denoted $\maxd{G}$ and 
$\mind{G}$, respectively. We denote the complement of $G$ as $\bar{G}$. 
For vertex $a$, $N(a)$ denotes the neighborhood of $a$, meaning the 
set of vertices adjacent to $a$.
For $a, b \in V(G)$, $G-a$ and $\Gab$ are the induced graphs
on $V(G) \setminus \{a\}$ and $V(G) \setminus \{a,b\}$. A graph is said to 
be \dfn{$n$-apex}
if there is a set of $n$ or fewer vertices whose deletion makes $G$ planar. 
This generalizes the notion of \dfn{apex}, common in the literature 
and which corresponds to $1$-apex.
The abbreviation MMNA describes
graphs that are minor minimal not apex.
Similarly MMN2A graphs are those that are minor minimal for the property
not $2$-apex.

For graph $G$ containing a $3$-cycle $abc$, a $\ty$ move 
results in a graph $G'$ of equal size to $G$ but with an additional vertex $v$.
The edges of the $3$-cycle are deleted and replaced with $av$, $bv$, and $cv$.
We say that $G'$ is a \dfn{child} of $G$.
The reverse operation (delete a degree three vertex $v$ and make $N(v)$ a 
$3$-cycle) is a $\yt$ move. We generally assume that any extra, doubled edges
introduced by a $\yt$ move are deleted so that the resulting graph is, again, simple.

In this paragraph we define several named graphs used here as well
as in the two theses~\cite{Mo,R} that are the basis of Sections 2 and 4.
Graphs $H_8$, $F_9$, and $H_9$ are MMIK graphs in the $K_7$ family and 
were named by Kohara and Suzuki~\cite{KS}.
We use $A_9$ and $B_9$ to denote the two children of $K_{3,3,1,1}$. 
These are called $L_1$ and $L_2$ in~\cite{OT}, and 
are Cousins 2 and 3 in the $K_{3,3,1,1}$ family as described in~\cite{GMN}.
As with all graphs in that family, $A_9$ and $B_9$ are MMIK.

The third author's thesis~\cite{R} includes an unknotted embedding 
(due to Ramin Naimi~\cite{N}) of 
a $(9,29)$ graph that we call 260910. The complement of 260910 is
the disjoint union of a $6$-cycle, $K_2$, 
and $K_1$. It's also given by adding the edges $\{1,2\}$, $\{2,3\}$, 
and $\{2,4\}$ to the graph $G_{9,26}$ described near the
end of Section 4.

\begin{lemma}
\label{lemd3}%
If $G$ is MMIK, then $\mind{G} \geq 3$.
\end{lemma}

\begin{proof}
Suppose $G$ is IK.
If $\mind{G} \leq 2$, we can form a proper minor $H$ either by deleting a vertex, or by contracting
an edge adjacent to a vertex of degree one or two. Then $H$ is also IK, and $G$ is 
not MMIK.
\end{proof}

\begin{lemma}
\label{lemtyIK}%
If $G$ is IK and $H$ is a child of $G$, then $H$ is also IK.
\end{lemma}

\begin{proof}
Sachs~\cite{S} showed this for the intrinsic linking property.
The proof for IK is similar.
\end{proof}

\begin{lemma} \cite{CMOPRW}
\label{lemCMOPRW}%
If $|G| = n > 7$ and $\|G\| \geq 5n-14$, then $G$ is IK but not MMIK.
\end{lemma}

\begin{proof} Mader~\cite{Ma} has shown that such a graph has a $K_7$ minor.
\end{proof}

\begin{lemma} \cite{BBFFHL, OT}
\label{lem2apex}%
If $G$ is IK, then $G$ is not $2$-apex
\end{lemma}

\begin{lemma} \cite{BM}
\label{lemapex}%
A graph that is not apex has at least 15 edges. The 
graphs in the Petersen family are the only MMNA graphs
of size 16 or less.
\end{lemma}

\section{Proof of Proposition~\ref{prop28}}
In this section we summarize the proof of Proposition~\ref{prop28},
see~\cite{R} for additional detail. We begin with a lemma.

\begin{lemma}
\label{lemd4}%
If $G$ is MMIK of order nine and $\|G\| = 29$ or $30$, then $\mind{G} \geq 4$.
\end{lemma}

\begin{proof}
Let $G$ be MMIK and either $(9,29)$ or $(9,30)$. By Lemma~\ref{lemd3},
$\mind{G} \geq 3$. Suppose there is a vertex of degree three and perform
the $\yt$ move on $G$ to obtain a graph $G'$ of order eight. 
After removing doubled edges, $26 \leq \|G'\| \leq 28$, and by Lemma~\ref{lemCMOPRW}, 
$G'$ is IK.  Reverse the $\yt$ move by applying a $\ty$ move to $G'$ to obtain 
the graph $H$, a proper subgraph of $G$. Then $H$ is IK by Lemma~\ref{lemtyIK},
contradicting $G$ MMIK. Therefore, $\mind{G} \geq 4$.
\end{proof}

\begin{proof} (of Proposition~\ref{prop28})
Let $|G| = 9$. By Lemma~\ref{lemCMOPRW}, if $\|G\| \geq 31$, then $G$ is not MMIK. This leaves three cases: $\|G\| =  30$,  $29$, and $28$. 

Suppose $G$ is a $(9,30)$ graph.
There are exactly 63 such graphs, four of which have $\mind{G} < 4$ and are not 
MMIK by Lemma~\ref{lemd4}. Of the remainder, 51 have the MMIK graph
$A_9$ as a subgraph. An additional five are $2$-apex and of the from $P+K_2$, 
the join with $K_2$ of a
planar triangulation on seven vertices. This leaves three graphs, two of which have 
$B_9$ as a subgraph and the last having a $K_7$ minor.

Next suppose $G$ is a $(9,29)$ graph. There are exactly 148 such graphs, of which 15 have 
$\mind{G} < 4$. The remainder include 25 graphs that are subgraphs of the
size 30 $P+K_2$ examples as well as a graph (called 260910 in~\cite{R}) that 
has an unknotted embedding although it is not $2$-apex. (The unknotted embedding 
given in~\cite{R} is due to Ramin Naimi~\cite{N}.) The other 107 graphs are IK
but not minor minimal as they admit either an $A_9$ subgraph (97 graphs) or else a $K_7$ or
$B_9$ minor (five each).

There are 344 connected $(9,28)$ graphs, of which 11 have $\mind{G} < 3$ and 39
$\mind{G} =3$. 
The analysis of \cite{R} is mainly concerned with the 294 graphs with
$\mind{G} > 3$. For example, the appendix of that paper includes drawings of all of those graphs.
As mentioned there, there are 181 IK graphs including
168 having $F_9$ subgraph, four with a $B_9$ subgraph, and a further eight with a $K_7$ minor. 
Although these 180 graphs are IK, the proper minors show that they are not MMIK. In addition, there
is a single MMIK graph, whose complement is the disjoint union of $K_2$ and a seven cycle. 
(See also \cite{GMN} where this graph is called $G_{9,28}$ and shown to be MMIK.)
The remaining 113 graphs are not IK, and all but two of these are $2$-apex. 
Those two are subgraphs of the $(9,29)$ graph 260910 which has an 
unknotted embedding.
There are a couple of typos in \cite{R}, which we correct here. There are 97 (and not 98) graphs
whose complement has three components. In case the complement consists of two isolated vertices and
a third connected component, there are 56 (and not 57) graphs with $F_9$ subgraph.

It remains to investigate the 39 graphs of minimum degree three. Let $G$ be a $(9,28)$ graph with
$\mind{G} = 3$. It's easy to see that there can be at most one degree three vertex, let's call it $a$,
and denote by $G'$ the result of a $\yt$ move at $a$. We delete any double 
edges so that $25 \leq \|G'\| \leq 28$. Suppose $G'$ is IK. Then applying a $\ty$ move to reverse the
$\yt$ move gives $H$, an IK subgraph of $G$ with $\|H \| = \|G'\|$. If $H$ is a proper subgraph of $G$,
then $G$ is not MMIK. The only other possibility is that $\|G'\| = 28$, meaning $G'$ is $K_8$. But
a $\ty$ move on $K_8$ results in a graph ($H = G$) that has a $F_9$ subgraph and is not MMIK.

So, we may assume $G'$ is not IK. By the classification of order eight IK graphs~\cite{BBFFHL,CMOPRW}, 
$G'$ is one of two graphs, both of which are $2$-apex. We'll argue $G$ must also be $2$-apex, hence not MMIK.
Note that losing three edges in the $\yt$ move means $a$ is part of a 
$K_4$ subgraph in $G$. 
In other words, $G$ is formed from $G'$ by adding, to a $3$--cycle 
$xyz$ of $G'$, vertex
$a$ and the three edges $ax$, $ay$, and $az$. The two $G'$ graphs
have, up to symmetry, three or five $3$--cycles respectively. It's easy to check
that adding $a$ and its edges to any of these results in a graph
$G$ that is again $2$-apex.
\end{proof}

\section{Proof of Proposition~\ref{prop22}}
\begin{proof} (of Proposition~\ref{prop22})
Suppose $G$ is a $(9,22)$ MMIK graph. 
By Lemma~\ref{lem2apex},
$G$ is not $2$-apex and, for any vertex $a$, $G-a$ is not apex.
By Lemma~\ref{lemapex}, $\|G-a\| \geq 15$, so $\maxd{G} \leq 7$.
Since a $4$--regular graph has only 18 edges, $\maxd{G} \geq 5$. 

Suppose
$\maxd{G} = 5$. Then $G$ has degree sequence $(5^8, 4)$ and there are
a  pair of nonadjacent degree five vertices $a$, $b$ so that $G-a,b$ is a
nonplanar $(7,12)$ graph with $\mind{G-a,b} \geq 2$. Moreover,
$G-a,b$ can have at most one degree two vertex. We next show that
we can assume $\maxd{G-a,b} = 5$.

Indeed, if not, then for every choice of vertices $a$, $b$ of degree five,
every other vertex of degree five is adjacent to at least one of $a$ and $b$.
This means that, in the complement, there are no triangles of degree five
vertices. Deleting the degree four vertex gives a $(8,18)$ graph.
The complement $\overline{G-a}$ is then a triangle free $(8,10)$ graph
of degree sequence $(3^4, 2^4)$. Contract an edge adjacent to each
degree two vertex to make a minor that is a $3$--regular $(4,6)$
multigraph, either $K_4$ or a $4$--cycle with opposing edges doubled.
Subdividing edges so as to eliminate triangles, we see that there
are only ten such $(8,10)$ graphs $\overline{G-a}$. Adding back
the degree four vertex, we find that each of the resulting graphs is $2$-apex,
contradicting that $G$ is MMIK.

\begin{figure}[htb]
\begin{center}
\includegraphics[scale=0.75]{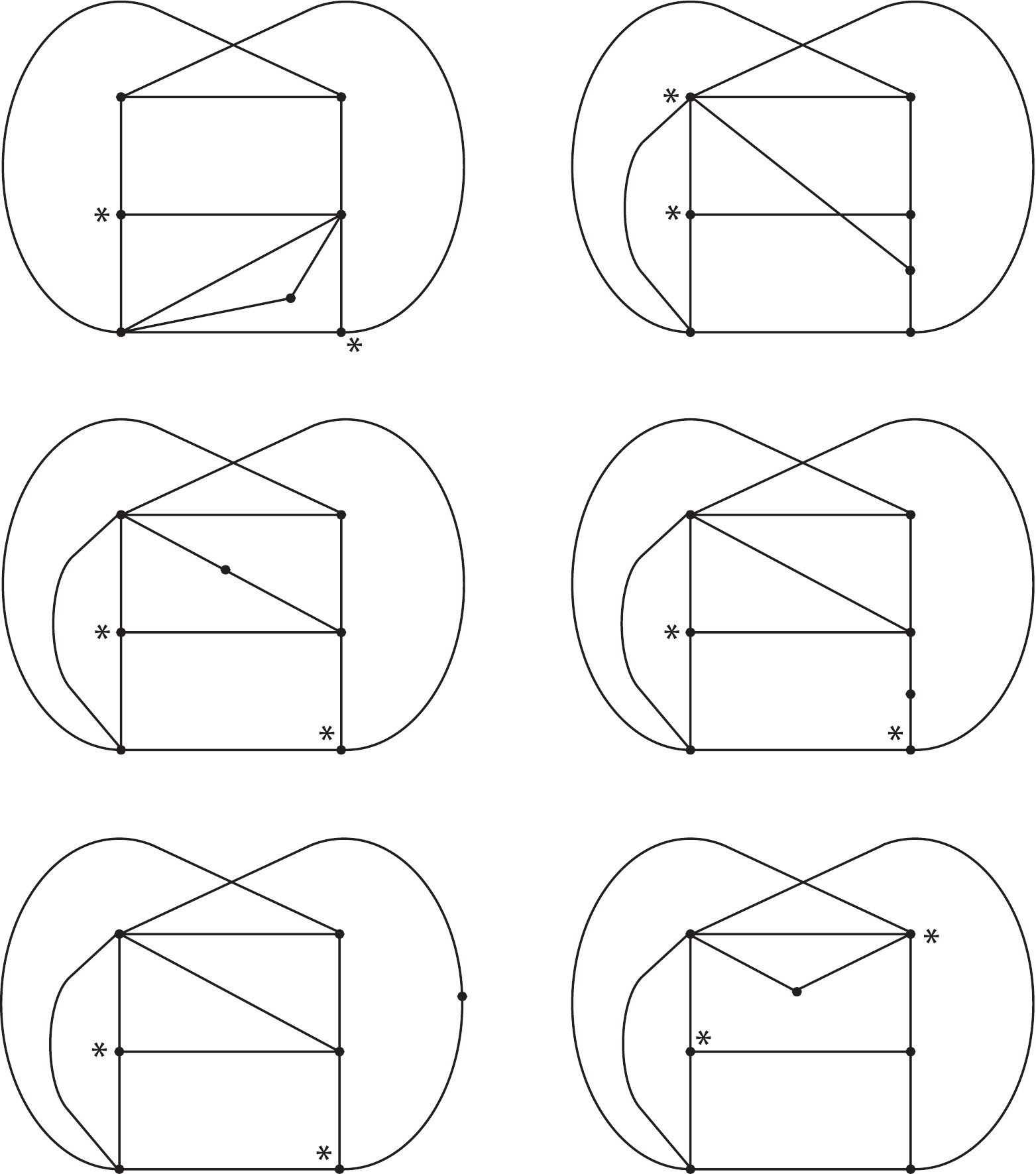}
\caption{Six non planar $(7,12)$ graphs.}
\label{fig7NP}%
\end{center}
\end{figure}

We assume, then, that $\maxd{\Gab} = 5$, $\mind{\Gab} \geq 2$, 
and that $\Gab$ is $(7,12)$ and nonplanar. 
There are six such graphs shown in Figure~\ref{fig7NP}.
In each case, deleting the two starred vertices in $G$ results in a planar graph. 
Therefore, $G$ is $2$-apex and not MMIK, a contradiction.

Next, suppose $\maxd{G} = 7$. Deleting a vertex $a$ of maximal degree leaves
an $(8,15)$ graph $G-a$ that is not apex. Since $G$ is MMIK, by
Lemma~\ref{lemd3}, $\mind{G-a} \geq 2$. Then, by 
Lemma~\ref{lemapex}, $G-a$ is one of the two $(8,15)$ Petersen family graphs,
which we denote $K_{4,4}-e$ and $P_8$. Since $a$ has degree seven,
it is adjacent to all but one vertex in $G-a$.

Suppose $G-a$ is $K_{4,4}-e$. There are two types of vertices, of degree
three and four, in $K_{4,4}-e$, and consequently $G$ is one of two 
graphs. If $a$ is adjacent to every vertex of $G-a$ but one of degree three, then
$G$ has a proper $H_9$ minor and is not MMIK. On the other hand, 
if it's a degree four vertex that is not a neighbor of $a$, then $G$ is $2$-apex
and again, not MMIK.

There are four types of vertices in $P_8$ and four graphs that can be constructed
by adding a degree seven vertex to it. One of them is $A_9$. The others are not 
MMIK, one because it has a proper $F_9$ minor, the other two being $2$-apex.

\begin{figure}[htb]
\begin{center}
\includegraphics[scale=0.5]{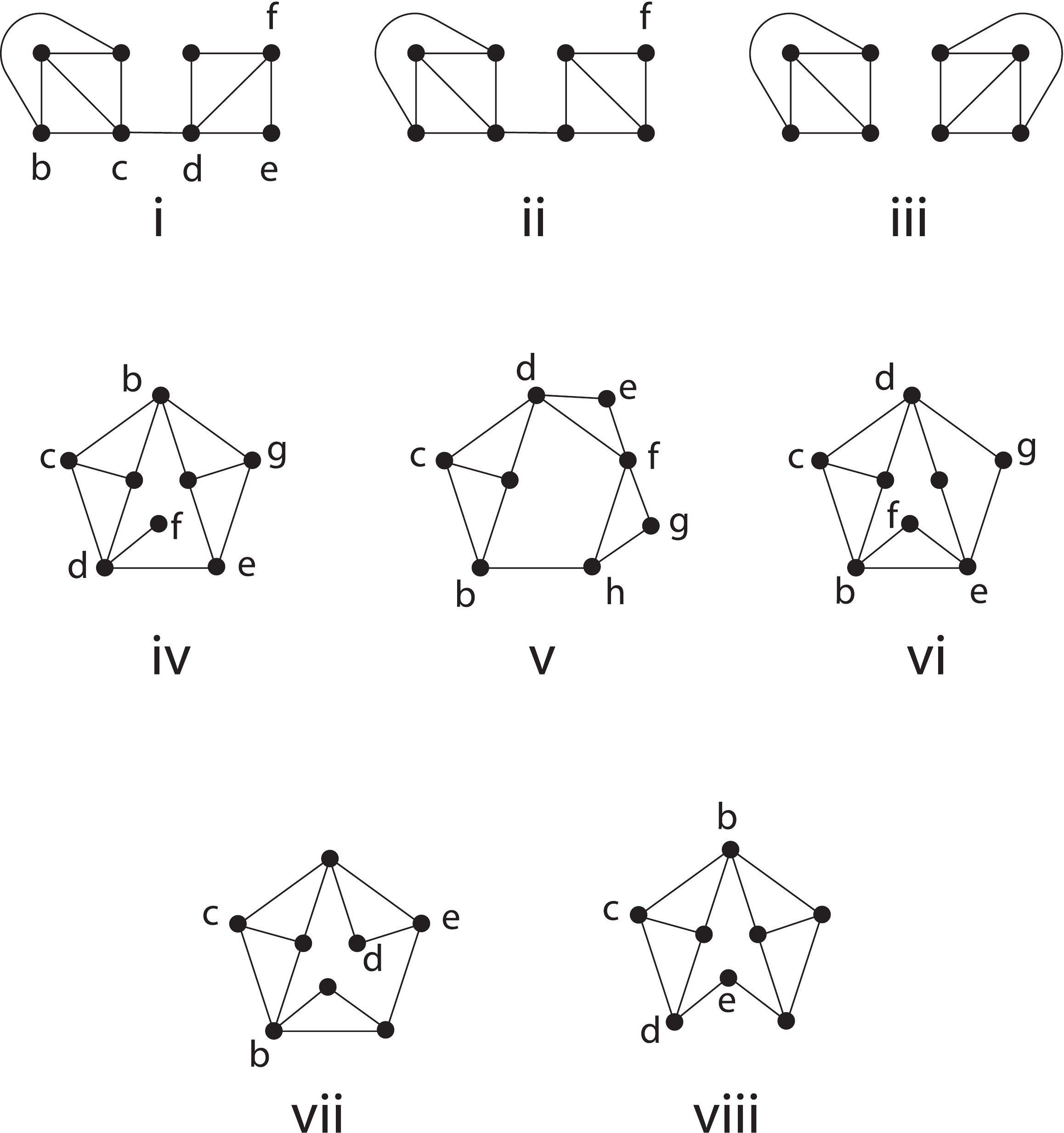}
\caption{Seven $(8,12)$ graphs formed by adding an edge to an order eight Petersen graph.}
\label{fig16NA1}%
\end{center}
\end{figure}

\begin{figure}[htb]
\begin{center}
\includegraphics[scale=0.5]{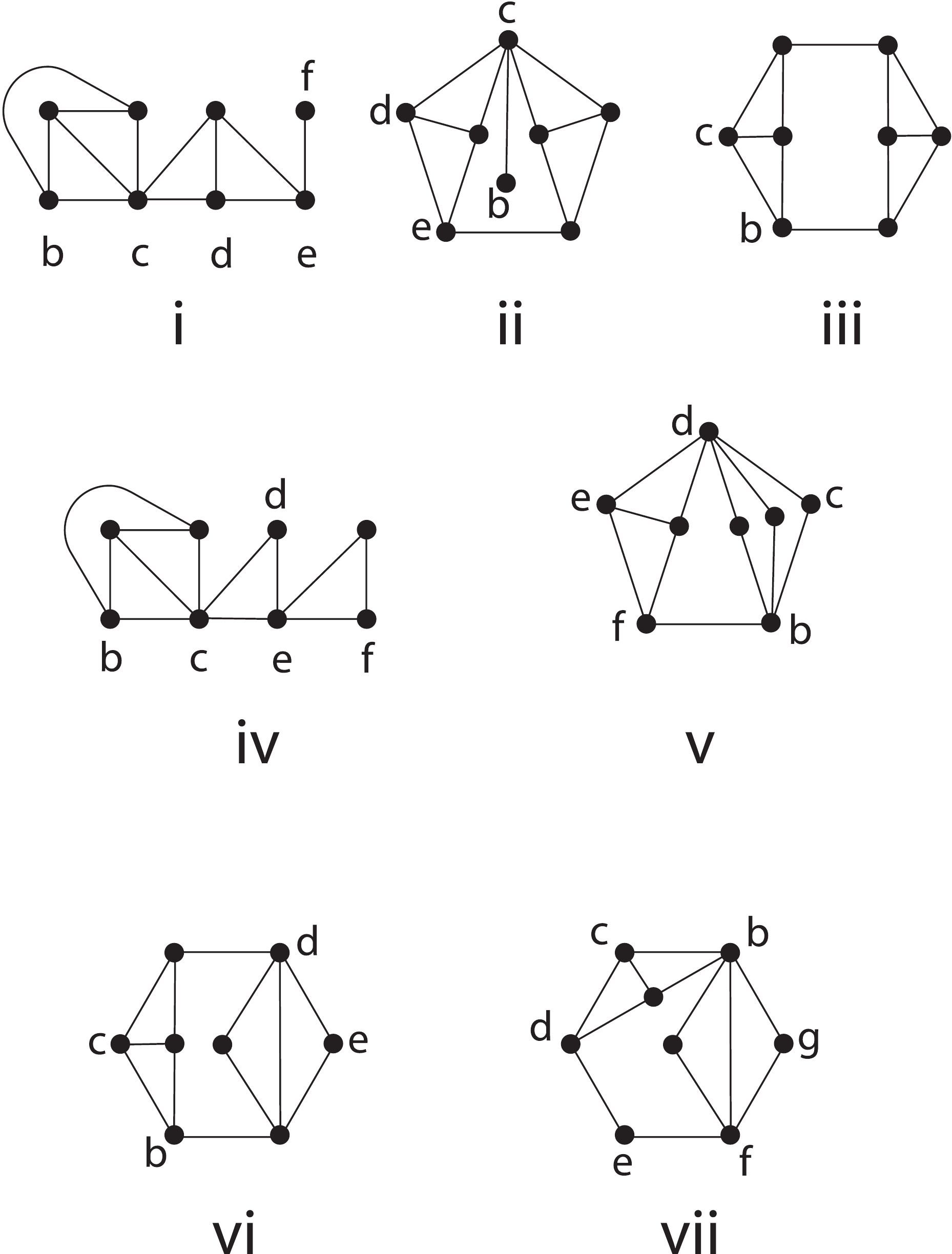}
\caption{Six $(8,12)$ graphs formed by splitting a vertex of an order seven Petersen graph.}
\label{fig16NA2}%
\end{center}
\end{figure}

In the remainder of the proof, we assume $\maxd{G} = 6$. For a vertex $a$ 
of maximal degree, $G-a$ is an $(8,16)$ not apex graph with $\maxd{G-a} \leq 6$ and
$\mind{G-a} \geq 2$. 
By Lemma~\ref{lemapex}, $G-a$ has a minor in the Petersen family.
Then $G-a$ is either one of eight graphs (see
Figure~\ref{fig16NA1} for the complements of these graphs) obtained by adding 
an edge to an eight vertex graph or one of seven (Figure~\ref{fig16NA2})
that results from splitting a vertex of an order seven Petersen family graph.

We conclude our proof by examining each of the 15 candidate $G-a$ graphs in turn.
To recover $G$, we add back the degree six vertex $a$, which is adjacent
to all but two of the vertices of $G-a$. So, in the figures, 
we label the different classes of vertex, up to symmetry, $b$, $c$, etc. 
In most cases, choosing the two vertices not in $N(a)$
results in a graph $G$ that is not IK because it is $2$-apex 
(Lemma~\ref{lem2apex})  and we'll only list the choices of vertex pair that 
produce an IK graph. Generally, this will be due to a proper IK minor,  but the 
four cases that result in a MMIK graph are Figure~\ref{fig16NA1}vii (yielding
Cousin 12 of the $K_{3,3,1,1}$ family~\cite{GMN}) and viii
($E_9+e$) as well as Figure~\ref{fig16NA2}i ($B_9$)
and iii (Cousin 41 of the $K_{3,3,1,1}$ family).

If $G-a$ is the graph of Figure~\ref{fig16NA1}i, 
there are five types of vertices. 
To get an IK graph, with proper $H_8$ minor, 
choose $\{c,d\}$ or $\{d,f\}$ as the vertex pair not in $N(a)$.

If $G-a$ is as in Figure~\ref{fig16NA1}ii,
there are again five types of vertices. 
The only case that results in an IK graph, with proper $H_9$ minor, is 
$\{c,d\} \cap N(a) = \emptyset$.

The graph of Figure~\ref{fig16NA1}iii is $K_{4,4}$ with all vertices symmetric.
However, the two that are not adjacent to $a$ are either in the same part,
in which case $G$ is $2$-apex, 
or in distinct parts, meaning $G$ has a proper $H_9$ minor. 

There are six types of vertices for the graph of Figure~\ref{fig16NA1}iv.
Except for the case where $\{d,e\} \cap N(a) = \emptyset$, which leads
to a proper $F_9$ minor, $G$ will be $2$-apex. 

There are only two vertices that share the same symmetry type (that 
of vertex $c$) for the graph of Figure~\ref{fig16NA1}v.
If $\{d,f\}$ or $\{f,h\}$ are the
pair not in $N(a)$, then $G$ has a proper $H_8$ minor. 

In the case of Figure~\ref{fig16NA1}vi, 
adding back in $a$ results in a $2$-apex graph except for two cases.
If $\{b,e\} \cap N(a) = \emptyset$, there is a proper $H_8$ minor and
if it's $d$ and $e$ that are avoided, a proper $F_9$ minor results.

For Figure~\ref{fig16NA1}vii,
there are only four types of vertices. To get an IK graph, use $\{b,e\}$ as the avoided
pair (which leads to a proper $F_9$ minor), or else the two vertices of type $e$. 
It is this last case that leads to an MMIK graph, Cousin 12 in the 
$K_{3,3,1,1}$ family of~\cite{GMN}.

For the final graph of Figure~\ref{fig16NA1},
avoiding the two vertices of type $d$ gives an IK graph that has a proper
$F_9$ minor. On the other hand, 
if $\{d,e\} \cap N(a)= \emptyset$, 
we obtain the MMIK graph $E_9+e$.

Turning to Figure~\ref{fig16NA2}, the complement of graph i
will produce a $2$-apex graph unless it's $\{e,f\}$ that are not in $N(a)$. 
In that case, we achieve the MMIK graph $B_9$.

For Figure~\ref{fig16NA2}ii,
the result is IK only in the case that $a$ is adjacent to neither vertex of type $e$, in 
which case we have the MMIK graph $A_9$. However, this graph has $\maxd{G}$ seven, 
and not six as we have been assuming.

The only way to get an IK graph from Figure~\ref{fig16NA2}iii,
by having $a$ avoid two vertices of type $b$, gives the MMIK graph
denoted Cousin 41 of of the $K_{3,3,1,1}$ family in~\cite{GMN}.

There are three ways to construct an IK graph from Figure~\ref{fig16NA2}iv.
If $\{c,e\}$ or $\{e,f\}$ are outside of $N(a)$, then $G$ will have a proper
$H_8$ minor. On the other hand, if $a$ is not adjacent to either of 
the vertices of type $f$, there will be a proper $H_9$ minor.

The IK graphs obtained from Figure~\ref{fig16NA2}v,
by avoiding $\{b,d\}$ or $\{b,f\}$, both have a proper $H_8$ minor.

The $G-a$ whose complement is Figure~\ref{fig16NA2}vi is similar.
There are two ways to achieve an IK graph, both with $H_8$ minor:
take either $\{b,d\}$ or the pair of vertices of type $d$ outside of
$N(a)$.

The final example, Figure~\ref{fig16NA2}vii,
can produce IK graphs in a number of ways. If $\{b, f\} \cap N(a)
=\emptyset$, $G$ has a proper $H_8$ minor.
If $N(a)$ avoids any two of $d$, $e$, and $f$, then $G$ has a proper $F_9$ minor.
\end{proof}

\section{Computer Verification for Size 23 through 27}
In this section we show that a graph $G$ with $|G| = 9$ and $23 \leq \|G\| \leq 27$ cannot be MMIK. We outline two approaches. The first is
found in the second author's thesis~\cite{Mo} where he implemented the algorithm in both Ruby and Java languages.
The second is based on a classification of MMN2A (minor minimal not $2$-apex) graphs~\cite{MP} achieved using Mathematica.

The idea of the first approach is to start with a listing of all graphs $G$ with 
$|G| = 9$ and $23 \leq \|G\| \leq 27$ and systematically apply a
sequence of six tests to $G$ in an effort to determine 
whether or not it is IK. In each case, applying the
test to a graph $G$ produces one of three outcomes: the graph $G$ is IK,
the graph $G$ is not IK, or the status of graph $G$ remains indeterminate.

The first three tests are based on basic facts about the order and size
of $G$. By assuming $G$ of order nine with $23 \leq \|G \| \leq 27$, 
we've already taken these constraints into consideration and these
three tests will leave $G$ indeterminate.
A fourth test, the Minor Of Classification, checks if $G$ is a minor
of certain well-known MMIK graphs. However, the implementation
only makes use of MMIK graphs on 21 and 22 edges and again
will leave our graphs, of size 23 at least, indeterminate. 

This leaves only two tests. Fortunately, they are quite effective at
sorting our graphs. The Contains Minor Classification
checks if $G$ has $K_7$, $H_8$, $F_9$, $H_9$, $K_{3,3,1,1}$,
$A_9$, or $B_9$ as a minor. If so $G$ is IK. As we are assuming $G$ has
at least 23 edges the found minor is proper and $G$ is not MMIK.
The final test,
Planarity Classification, says $G$ is not IK, hence not MMIK, if it is $2$-apex.

In \cite{Mo}, the algorithm is applied to all connected graphs of order nine, 
leaving only 32 indeterminate graphs with 24 having $23 \leq \|G\| \leq 27$. Four of
these are subgraphs of the graph 260910. Since 260910 has an 
unknotted embedding (see~\cite{R}) these four are not IK. The remaining 20 
indeterminate graphs have $E_9+e$, a 22 edge MMIK graph, as a proper minor 
and are therefore not MMIK themselves.

A second proof that there is no MMIK order nine graphs with $23 \leq \|G\| \leq 27$
is based on the classification of order nine MMN2A graphs.
Using Mathematica, in~\cite{MP} the authors show that there are a total of 12 MMN2A graphs
through order nine, including five each in the $K_7$ and $K_{3,3,1,1}$ families as well
as a $(9,26)$ and $(9,27)$ example; we'll call them $G_{9,26}$ and $G_{9,27}$.

By Lemma~\ref{lem2apex}, a MMIK graph of order nine must have a MMN2A minor. 
Most of the 12 MMN2A graphs are in fact MMIK. As shown in~\cite{GMN} every graph 
in the $K_{3,3,1,1}$ family is MMIK. Also, all but six of the graphs in the $K_7$ family 
are MMIK~\cite{GMN,HNTY}. Only one of the six exceptions has order nine or less; 
in~\cite{GMN} they call that graph $E_9$. So, to show that there are no MMIK graphs
of order nine with $23 \leq \|G \| \leq 27$, it's enough to prove the following 
proposition.

\begin{prop}
There is no MMIK graph $G$ of order nine with $23 \leq \|G \| \leq 27$ that has
$E_9$, $G_{9,26}$, or $G_{9,27}$ as a minor.
\end{prop}

\begin{proof}
The three MMN2A graphs mentioned in the statement, $E_9$, $G_{9,26}$, and $G_{9,27}$, are 
all of order nine. If $G$ is an order nine MMIK graph with one of these three as a minor, 
then $G$ is formed by adding edges. 

Up to symmetry, there are two types of edges missing from $E_9$. 
By adding a single edge, we either form $E_9+e$, a MMIK graph of order 22 (see~\cite{GMN}), or else a 22 edge graph with
a proper $F_9$ subgraph. To construct
$G$ with $23 \leq \|G \| \leq 27$ requires addition of further edges and any such $G$ has either $E_9+e$
or $F_9$ as a proper minor and is not MMIK.

We can describe $G_{9,26}$ by its edge list~\cite{MP}:
$$\{\{1, 4\}, \{1, 5\}, \{1, 7\}, \{1, 8\}, \{1, 9\}, \{2, 5\}, \{2, 6\}, \{2, 7\}, \{2,   8\},$$
$$   \{2, 9\}, \{3, 5\}, \{3, 6\},  \{3, 7\}, \{3, 8\}, \{3, 9\}, \{4, 6\}, \{4,   7\}, \{4, 8\}, $$
$$  \{4, 9\}, \{5, 6\}, \{5, 8\}, \{5, 9\}, \{6, 8\}, \{6, 9\}, \{7,   8\}, \{7, 9\}\} .$$
Adding edges $\{1,3\}$ and $\{2,4\}$ results in the graph $G_{9,28}$, which is MMIK~\cite{GMN}. 
This means neither $G_{9,26}$, nor any 27 edge graph formed by adding a single one of those edges 
is MMIK or even IK. It remains to investigate adding other edges to $G_{9,26}$. Up to symmetry, there are four 
other possibilities. 
Adding $\{2,3\}$ results in a graph 
that is not IK since it is a subgraph of 
the graph 260910, which has an unknotted embedding (see~\cite{R}).
Adding $\{8,9\}$ or $\{5,7\}$ gives a size 27 graph with a proper $A_9$ minor, while
the graph formed by adding $\{1,6\}$ has a proper $B_9$ minor. So, neither $G_{9,26}$ nor
any 27 edge graph formed by adding a single edge is MMIK.

The remaining possibility is that $G_{9,27}$ itself is MMIK. However, $G_{9,27}$ is a subgraph of 
the graph 260910 that is shown to have an unknotted embedding in~\cite{R}. Therefore, $G_{9,27}$ is
not IK and also not MMIK.
\end{proof}

In summary, we've described two different arguments that there is no order nine MMIK graph $G$
with $23 \leq \|G\| \leq 27$. The first approach is implemented in two different languages in~\cite{Mo}. 
The second approach instead relies on the classification of order nine 
MMN2A graphs achieved using Mathematica~\cite{MP}.
\section*{Acknowledgements}

\end{document}